\documentclass[11pt,a4paper]{article}
\usepackage{amsfonts}
\usepackage{latexsym}
\usepackage{amsmath}
\usepackage{graphicx,epsfig}
\usepackage{amstext}
\usepackage[spanish,english]{babel}
\usepackage{amssymb}

\newtheorem{theorem}{Theorem}

\newtheorem{example}[theorem]{Example}
\newtheorem{lemma}[theorem]{Lemma}

\newtheorem{remark}[theorem]{Remark}

\newenvironment{proof}{{\it Proof.}}{\hfill $\square $}

\title{Note on the location of zeros of polynomials}
\author{ J. Rubi\'{o}--Masseg\'{u}
  \\*[.1truecm]
{\small \textsl{Departament de Matem\`{a}tica Aplicada III,}}
\\*[-.25truecm] {\small \textsl{Control, Din\`{a}mica i Aplicacions (CoDALab)}}
\\*[-.25truecm] {\small \textsl{Universitat Polit\`{e}cnica de Catalunya}}
\\*[-.25truecm] {\small \textsl{Avinguda de les Bases de Manresa, 61--73, 08242 Manresa, Spain}}
\\*[-.25truecm] {\small \textsl{E--mail address: josep.rubio@upc.edu}}}

\begin{document}

\date{}
\maketitle

\begin{abstract}
\noindent In this note, we provide a wide range of upper bounds for the moduli
of the zeros of a complex polynomial. The obtained bounds complete a series of previous papers on the location of
zeros of polynomials.
\end{abstract}

\vskip10pt
\noindent \textit{Keywords}: Complex polynomials; Location of zeros of polynomials; Cauchy's bound

\noindent \textit{2000MSC}: 26C10, 30C15, 65H05

\section{Introduction}

The theory of the location of zeros of polynomials has applications
in several areas of contemporary applied mathematics, including
linear control systems, electrical networks, root approximation,
signal processing and coding theory. 
Because of its applications, there is a need for obtaining better
and better results in this subject. A review on the location of
zeros of polynomials can be found in \cite{M1966,MMR1994,RS2002}.

In what follows, $P(z)$ is the complex polynomial
\begin{equation}\label{P}
P(z)=z^{n}+a_{1}z^{n-1}+\cdots +a_{n-1}z+a_{n}.
\end{equation}
Without loss of generality we will assume that $a_{j}\neq 0$ for at
least one $j$, and set $a_{j}=0$ for $j>n$.  According to a result
of Cauchy \cite{Cauchy1829}, all the zeros of the polynomial $P(z)$
are in the circle $|z|\leq \rho$, where $\rho $ is the unique
positive zero of the real polynomial
\begin{equation}\label{Q}
Q(x)=x^{n}-|a_{1}|x^{n-1}-\cdots -|a_{n-1}|x-|a_{n}|.
\end{equation}
The upper bound $\rho$ is the best possible one which is
expressible in terms of the moduli of the coefficients. 
A classical result due to Cauchy \cite{Cauchy1829}
states that all the zeros of the polynomial $P(z)$ are contained in
the disk
\begin{equation}\label{Cauchy-bound}
|z|\leq \rho <1+A
\end{equation}
with $A=\max_{1\leq j\leq n}|a_{j}|$. As an improvement, Joyal, Labelle and Rahman \cite{JLR1967} proved
the following theorem.

\begin{theorem} \label{TheoremJLR} All the zeros of $P(z)$ are contained
in the disk
\begin{equation}\label{JLR-bound}
|z|\leq \rho \leq
\frac{1}{2}(|a_{1}|+1+\sqrt{(|a_{1}|-1)^{2}+4A_{2}})
\end{equation}
with $A_{2}=\max_{2\leq j\leq n}|a_{j}|$.
\end{theorem}

For each positive integer number $\ell \geq 1$, let $Q_{\ell }(x)$
be the polynomial
\begin{equation}\label{Ql}
Q_{\ell }(x)= x^{\ell }+\sum_{v=2}^{\ell }\left[ C_{\ell -v}^{\ell
-1}-\sum_{j=1}^{v-1}C_{\ell -v}^{\ell -j-1}|a_{j}|\right] x^{\ell
+1-v},
\end{equation}
where $C_{s}^{m}$ ($0\leq s\leq m$) are the binomial coefficients
defined by $C_{s}^{m}=\frac{m!}{s!(m-s)!}$. More recently, Affane--Aji et
al. \cite{AAG2009} have obtained the following result.

\begin{theorem} \label{TheoremAAG} All the zeros of $P(z)$ satisfy $|z|<1+\delta _{\ell }$,
for $\ell =1,2,\ldots $, where $\delta _{\ell }$ (for $\ell $ a
positive integer) is the unique positive root of the $\ell $th
degree equation
\begin{equation}\label{Ql-equation}
Q_{\ell }(x)=A.
\end{equation}
Moreover, $1+\delta _{1}\geq 1+\delta _{2}\geq \cdots \geq 1+\delta
_{\ell }>\max (1,\rho )$, for all $\ell \geq 1$.
\end{theorem}

Theorem \ref{TheoremAAG} provides a tool for obtaining sharper
bounds for the location of the zeros of a polynomial. When $\ell =1$
it reduces to (\ref{Cauchy-bound}), and for $\ell =2$ it yields
\begin{equation}\label{AAGl=2}
|z|<1+\delta _{2}=\frac{1}{2}(|a_{1}|+1+\sqrt{(|a_{1}|-1)^{2}+4A}),
\end{equation}
which looks like (\ref{JLR-bound}) but never sharpens it. The cases
$\ell =3$ and $\ell =4$ give rise to cubic and quartic equations
which can be explicitly solved, and they are due to Sun and Hsieh
\cite{SH1996} and Jain \cite{J2006} respectively.

Observe that there is no way to establish a link between Theorem
\ref{TheoremJLR} and Theorem \ref{TheoremAAG}, except for the case
$\ell=2$, in which Theorem \ref{TheoremJLR} provides a better bound.
For example, if $P(z)=z^{n}+az^{n-1}$ with $ |a|>0 $, then Theorem
\ref{TheoremJLR} yields $|z|\leq \max (1,|a|)$. But $\rho =|a|$, so
$1+\delta _{\ell }>\max (1,|a|)$ for any $\ell $. Hence, in this case,
the bound obtained from Theorem \ref{TheoremJLR} is better than any
bound obtained from Theorem \ref{TheoremAAG}, although the last ones
may require a high computational cost in order to be obtained.

The aim of this note is to fill this gap by showing that Theorem
\ref{TheoremAAG} can be sharpened in a natural way, so that it
contains Theorem 1 as a particular case (see Theorem
\ref{mainresult}). Numerical examples will show that our improvement
provides bounds which may be considerably better than the preceding
ones.

The main result, Theorem \ref{mainresult}, completes the series of
papers \cite{AAG2009,SH1996,J2006} on the location of zeros of
polynomials, and fills the gaps between them and Theorem
\ref{TheoremJLR}, due to Joyal et al. \cite{JLR1967}.

\section{The main result}
In what follows, we denote by $\varepsilon_{\ell}$ the largest real
root of the $\ell$th degree equation
\begin{equation}\label{mainresult-equation}
Q_{\ell }(x)=A_{\ell}
\end{equation}
with $A_{\ell}=\max_{j\geq \ell}|a_{j}|$, $\ell\geq 1$,
($A_{\ell}=0$ for $\ell >n$). Additionally, $q\in \{1,\ldots ,n\}$
is defined by the conditions
\begin{equation}\label{q-conditions}
a_{q}\neq 0,\hskip5pt \text{ and }a_{j}=0\hskip5pt \text{ for }j>q.
\end{equation}

\begin{theorem} \label{mainresult}
We have
\begin{equation*}
1+\varepsilon _{1}\geq 1+\varepsilon _{2}\geq \cdots \geq
1+\varepsilon _{q}>\max (1,\rho )=1+\varepsilon _{q+1}=1+\varepsilon
_{q+2}=\cdots
\end{equation*}
In particular, all the zeros of the polynomial $P(z)$ satisfy
$|z|<1+\varepsilon _{\ell }$, for $\ell =1,\ldots ,q$, and $|z|\leq
1+\varepsilon _{\ell }$ for $\ell>q$. Furthermore, $\varepsilon _{\ell}$
(for $1\leq \ell \leq q$) is the unique positive root of the equation $Q_{\ell }(x)=A_{\ell}$.
\end{theorem}

Observe that when $\ell =1$, the bound obtained from this result reduces to
(\ref{Cauchy-bound}), as Theorem \ref{TheoremAAG} did. However, for
$\ell =2$ now we obtain the bound (\ref{JLR-bound}), thus meeting
Theorem \ref{TheoremJLR} as desired. Finally, if $\ell>q$, we have
$|z|\leq 1+\varepsilon _{\ell }=\max (1,\rho )$, which reduces to
$|z|\leq \rho$ when $\rho \geq 1$.

Next, we prove that Theorem \ref{mainresult} sharpens Theorem
\ref{TheoremAAG}. To this end, we shall show that $\varepsilon
_{\ell }\leq \delta _{\ell }$ for any $\ell $. Indeed, if $\ell >q$
then $1+\varepsilon _{\ell }=\max (1,\rho )<1+\delta _{\ell }$,
hence $\varepsilon _{\ell }<\delta _{\ell }$. For $1\leq \ell \leq
q$, we consider the polynomial $R(x)= Q_{\ell }(x)-A_{\ell }$. Since
$ R(0)=-A_{\ell }<0$ then $R(x)<0$ for $0<x<\varepsilon _{\ell }$,
and taking into account that $R(\delta _{\ell })=A-A_{\ell }\geq 0$,
we obtain $\delta _{\ell }\geq \varepsilon _{\ell }$, and we are
done. Observe that $\varepsilon _{\ell }< \delta _{\ell }$ when
$A_{\ell}<A$.

\begin{remark}
By looking at equations (\ref{Ql-equation}) and
(\ref{mainresult-equation}), the algorithm using MATLAB that has
been successfully developed in \cite{AAG2009}, and which has as its
output the upper bound $1+\delta_{\ell}$, may be used to obtain an
algorithm having as its output the upper bound
$1+\varepsilon_{\ell}$, by simply replacing number $A$ in that
algorithm by number $A_{\ell}$. This simple modification should be
taken under consideration, since the bounds obtained from Theorem
\ref{mainresult} may be considerably better than the ones obtained
from Theorem \ref{TheoremAAG}.
\end{remark}

\begin{example}
For the polynomial
$P(z)=z^{5}+3z^{4}+2z^{2}+2$,  we have $\rho =3.21256$ and it
coincides with the largest modulus of the zeros. On the other
hand, we have
\begin{equation*}
\begin{tabular}{|c|c|c|}
\hline $\ell $ & $1+\varepsilon _{\ell }$ & $1+\delta _{\ell }$ \\
\hline $1$ & $4.00000$ & $4.00000$ \\ \hline $2$ & $3.73205$ & $4.00000$ \\
\hline $3$ & $3.26953$ & $3.37442$ \\ \hline $4$ &
$3.26953$ & $3.30278$ \\ \hline $5$ & $3.21989$ & $3.23138$ \\
\hline $6$ & $3.21256(=\rho)$ & $3.22350$
\\ \hline
\end{tabular}
\end{equation*}

\end{example}

\begin{example}
Let $P(z)=z^{10}+2z^{9}-3z^{8}+2z^{5}-z^{4}+z+2$. Then $\rho =3.02120$
and the largest modulus of the zeros is $3.02106$. We have
\begin{equation*}
\begin{tabular}{|c|c|c|}
\hline $\ell $ & $1+\varepsilon _{\ell }$ & $1+\delta _{\ell }$ \\
\hline $1$ & $4.00000$ & $4.00000$ \\ \hline $2$ & $3.30278$ & $3.30278$ \\
\hline $3$ & $3.21432$ & $3.30278$ \\ \hline $4$ & $3.07678$ &
$3.11111$ \\ \hline $5$ & $3.02675$ & $3.03942$
\\ \hline $10$ & $3.02124$ & $3.02129$
\\ \hline $11$ & $3.02120$ & $3.02125$
\\ \hline
\end{tabular}
\end{equation*}

\end{example}

\begin{example} For the polynomial $P(z)=z^{20}-0.6z^{19}-0.3z^{15}-0.2z^{8}-0.1z-0.2$, we
have $\rho =1.05673$, which coincides with the largest modulus of
the zeros. The bounds are
\begin{equation*}
\begin{tabular}{|c|c|c|}
\hline $\ell $ & $1+\varepsilon _{\ell }$ & $1+\delta _{\ell }$ \\
\hline $1$ & $1.60000$ & $1.60000$ \\ \hline $2$ & $1.38310$ & $1.60000$ \\
\hline $3$ & $1.31742$ & $1.46954$ \\ \hline $4$ & $1.27413$ &
$1.39150$ \\ \hline $5$ & $1.24297$ & $1.33864$
\\ \hline $6$ & $1.20500$ & $1.31930$
\\ \hline $10$ & $1.15805$ & $1.22986$
\\ \hline $21$ & $1.05673$ & $1.14649$
\\ \hline
\end{tabular}
\end{equation*}

\end{example}

\section{Proof of the main result}

In order to prove Theorem \ref{mainresult}, first we will prove a
result that will be
shown to be equivalent to the main result, and which is interesting
in itself because it simplifies considerably the expression of the
equation to be solved.

For $\ell\geq2$ an integer number, let $r_{\ell}$ be the largest
real zero of the $\ell$th degree polynomial
\begin{equation}\label{Pl}
P_{\ell }(x)= x^{\ell }-(|a_{1}|+1)x^{\ell -1}-\sum_{j=2}^{\ell
-1}(|a_{j}|-|a_{j-1}|)x^{\ell -j}-(A_{\ell }-|a_{\ell -1}|),
\end{equation}
where $a_j=0$ for $j>n$. When $\ell=1$ we define $r_1=1+A$, the
unique zero of the polynomial $P_{1}(x)= x-(1+A)$. In what follows,
$q$ is defined by the conditions (\ref{q-conditions}). Then,

\begin{theorem} \label{mainresult2}
We have
\begin{equation}\label{chain-mainresult2}
r_{1}\geq r_{2}\geq \cdots \geq r_{q}>\max (1,\rho
)=r_{q+1}=r_{q+2}=\cdots
\end{equation}
In particular, all the zeros of the polynomial $P(z)$ satisfy
$|z|<r_{\ell }$, for $\ell =1,\ldots ,q$, and $|z|\leq r_{\ell }$
for $\ell>q$. Furthermore, $r_{\ell}$
(for $1\leq \ell \leq q$) is the unique zero of the
polynomial $P_{\ell}(x)$ in the interval $[1,+\infty )$.
\end{theorem}
\begin{proof} By dividing the polynomial $P_{\ell }(x)$
by $x-1$, we have
\begin{equation}\label{Pl-equation}
P_{\ell }(x)= (x-1)F_{\ell }(x)-A_{\ell},
\end{equation}
where
\begin{equation}\label{Fl}
F_{\ell }(x)= x^{\ell -1}-|a_{1}|x^{\ell -2}-\cdots -|a_{\ell
-2}|x-|a_{\ell -1}|,\hskip10pt \text{for }\ell \geq 2,
\end{equation}
and $F_{1}(x)= 1$.

We need a lemma which is part of the statement of Theorem
\ref{mainresult2}.

\begin{lemma} \label{Lemma1}
If $1\leq \ell \leq q$, then $P_{\ell }(x)$ has a unique zero in $
[1,+\infty )$.
\end{lemma}
\begin{proof} Let $\ell\in\{1,2,\ldots,q\}$ be given. Since
$P_{\ell }(1)=-A_{\ell }<0$ and $P_{\ell }(x)$ tends to $+\infty$
when $x$ tends to $+\infty $, there exists at
least one zero of $P_{\ell }(x)$ in $[1,+\infty )$. Let $\alpha$ be the largest real
zero of $F_{\ell }(x)$ (for $\ell=1$ we set $\alpha=0$). By Descarte's rule of signs, if
some coefficient $a_{j}$ for $1\leq j\leq \ell -1$ is nonzero, then
$\alpha$ is the unique strictly positive zero of $F_{\ell }(x)$;
otherwise $\alpha=0$.

Set $\mu=\max (1,\alpha)$, and we claim that $P_{\ell }(x)$ is
strictly increasing for $x\geq \mu$. In fact, if $\alpha=0$ (i.e.
$F_{\ell }(x)$ has the form $F_{\ell }(x)= x^{\ell -1}$), then
$P_{\ell}(x)=(x-1)x^{\ell -1}-A_{\ell }$, an increasing function for
$x\geq 1=\mu$. Assume that $\alpha>0$, and let
\begin{equation*}
F_{\ell }(x)\equiv (x-\alpha)S(x)
\end{equation*}
with $S(x)$ a real monic polynomial. From $\alpha>0$ and the fact
that the sequence of nonzero coefficients of $F_{\ell }(x)$
has only one change of sign, it follows that all the coefficients of
$ S(x)$ must be nonnegative; in particular, $S(x)$ is monotonically increasing and
positive for $x\geq 0$. This implies that $P_{\ell }(x)$ is strictly
increasing for $x\geq \mu$, as
\begin{equation*}
P_{\ell }(x)\equiv (x-1)(x-\alpha)S(x)-A_{\ell }.
\end{equation*}
The claim is proved.

Since $P_{\ell }(\mu)=-A_{\ell}<0$ and $P_{\ell }(x)$ is strictly
increasing for $x\geq \mu$, $P_{\ell}(x)$ has a unique zero in
the interval $[\mu,+\infty )$. Therefore, when $\mu=1$ it is the
unique zero in $[1,+\infty)$, and we are done. Finally, assume that
$\mu=\alpha>1$, and we shall see that $P_{\ell}$ is zero--free in
the interval $[1,\mu)$. Indeed, if $1\leq y<\mu=\alpha$ then
$F_{\ell }(y)<0$ and $P_{\ell }(y)=(y-1)F_{\ell }(y)-A_{\ell }\leq
-A_{\ell }<0$, hence $P_{\ell }(y)\neq 0$ for $y\in[1,\mu)$, as
desired. This completes the proof of the lemma.
\end{proof} 

Taking into account the preceding lemma and that $\rho$ is an upper
bound for the moduli of the zeros of $P(z)$, in order to prove
Theorem \ref{mainresult2} it suffices to prove the chain of
inequalities and equalities of (\ref{chain-mainresult2}). First, we
show that
\begin{equation}\label{aux1}
 r_{\ell }=\max (1,\rho) \hskip10pt \text{
for }\ell >q.
\end{equation}
In fact, when $\ell >q$ the definition of $F_{\ell}(x)$ in
(\ref{Fl}) yields $F_{\ell }(x)=x^{\ell-q-1}Q(x)$, where $Q(x)$ is
the polynomial of (\ref{Q}). Since $A_{\ell }=0$ for $\ell
>q$, by (\ref{Pl-equation}) we have
\begin{equation}\label{aux2}
P_{\ell }(x)\equiv x^{\ell-q-1}(x-1)Q(x),
\end{equation}
hence the strictly positive zeros of $P_{\ell }(x)$ are $x=1$ and $x=\rho$, thus proving
(\ref{aux1}).

Finally, we show that $r_{1 }\geq r_{2}\geq \cdots \geq r_{q}> \max
(1,\rho )$. Let $\ell\in \{1,\ldots,q\}$ be given, and we will see
that $r_{\ell+1}\leq r_{\ell}$, with the inequality strict if
$\ell=q$. Set $y=r_{\ell}$. Since
$0=P_{\ell}(y)=(y-1)F_{\ell}(y)-A_{\ell}$, then
$(y-1)F_{\ell}(y)=A_{\ell}$. On the other hand, an easy computation
shows that $F_{\ell +1}(x)= xF_{\ell}(x)-|a_{\ell}|$ for all $x$.
Therefore,
\begin{eqnarray}\label{aux3}
P_{\ell+1}(y) &=&(y-1)F_{\ell+1}(y)-A_{\ell+1} \notag\\
&=&(y-1)(yF_{\ell}(y)-|a_{\ell}|)-A_{\ell+1} \notag\\
&=&y(y-1)F_{\ell}(y)-(y-1)|a_{\ell}|-A_{\ell+1} \notag\\
&=&yA_{\ell}-(y-1)|a_{\ell}|-A_{\ell+1} \notag\\
&=&y(A_{\ell}-|a_{\ell}|)+|a_{\ell}|-A_{\ell+1},
\end{eqnarray}
and from $A_{\ell+1}\leq A_{\ell}$ we have
\begin{eqnarray*}
P_{\ell+1}(y) &\geq &y(A_{\ell}-|a_{\ell}|)+|a_{\ell}|-A_{\ell} \notag\\
&=&(y-1)(A_{\ell}-|a_{\ell}|) \notag\\
&\geq &0.
\end{eqnarray*}
When $\ell <q$, this implies that $r_{\ell}=y\geq r_{\ell+1}$, as we
wanted to see. If $\ell =q$, from (\ref{aux3}) and the fact that
$A_{q}=|a_{q}|>0=A_{q+1}$, it follows that $P_{q+1}(y)=|a_{q}|>0$.
Using the expression of $P_{\ell}(x)$ in (\ref{aux2}), which is
valid for $\ell>q$, we obtain that $P_{q+1}(y)= (y-1)Q(y)>0$, so
$Q(y)>0$ since $y=r_q>1$. This implies that $y>\rho$, hence
$r_{q}=y>\max(1,\rho)=r_{q+1}$, and Theorem \ref{mainresult2}
follows.
\end{proof}

Theorem \ref{mainresult} is a straightforward consequence of the
following lemma.

\begin{lemma}\label{Lemma2}
For any integer number $\ell \geq 1$, we have
\begin{equation*}
P_{\ell }(1+x)= Q_{\ell }(x)-A_{\ell}, \hskip5pt \text{ for all }x,
\end{equation*}
where $Q_{\ell }(x)$ is the polynomial of (\ref{Ql}).
\end{lemma}
\begin{proof} By (\ref{Pl-equation}) it is only necessary to see that
$xF_{\ell}(1+x)= Q_{\ell }(x)$. For $\ell=1$ it is clear. Set
$b_{0}=1$ and $b_{j}=-|a_{j}|$ for $j\geq 1$. Then $F_{\ell
}(x)=\sum_{k=0}^{\ell -1}b_{\ell -k-1}x^{k}$ and
\begin{eqnarray*}
xF_{\ell }(1+x) &=&x\sum_{k=0}^{\ell -1}b_{\ell -k-1}(1+x)^{k}\notag\\
&=&x\sum_{k=0}^{\ell -1} b_{\ell
-k-1}\sum_{i=0}^{k}C_{i}^{k}x^{i} \notag\\
&=&\sum_{i=0}^{\ell -1}\left( \sum_{k=i}^{\ell -1}C_{i}^{k}b_{\ell
-k-1}\right) x^{i+1} \notag\\
&=&x^{\ell }+\sum_{i=0}^{\ell -2}\left( \sum_{k=i}^{\ell
-1}C_{i}^{k}b_{\ell -k-1}\right) x^{i+1}.
\end{eqnarray*}

The substitutions $i=\ell -v$, ($2\leq v \leq \ell$), and $k=\ell
-j-1$, ($0\leq j \leq v-1$), yield
\begin{eqnarray*}
xF_{\ell }(1+x) &=& x^{\ell }+\sum_{v=2}^{\ell }\left( \sum_{k=\ell
-v}^{\ell
-1}C_{\ell -v}^{k}b_{\ell -k-1}\right) x^{\ell +1-v} \notag\\
&=&x^{\ell }+\sum_{v=2}^{\ell }\left(
\sum_{j=0}^{v-1}C_{\ell -v}^{\ell -j-1}b_{j}\right) x^{\ell +1-v} \notag\\
&=&x^{\ell }+\sum_{v=2}^{\ell }\left( C_{\ell -v}^{\ell
-1}-\sum_{j=1}^{v-1}C_{\ell -v}^{\ell -j-1}|a_{j}|\right) x^{\ell +1-v} \notag\\
&=&Q_{\ell }(x),
\end{eqnarray*}
and lemma follows.
\end{proof}

\begin{proof} [Proof of Theorem \ref{mainresult}]
The change of variable $x=y-1$,
$y\geq 1$, transforms equation (\ref{mainresult-equation}) into
the equation $P_{\ell}(y)=0$. Therefore, numbers
$\varepsilon_{\ell}$ and $r_{\ell}$ of Theorems \ref{mainresult} and
\ref{mainresult2} are related by $r_{\ell}=1+\varepsilon_{\ell}$. Now, applying Theorem \ref{mainresult2} we immediately obtain
Theorem \ref{mainresult}, and this completes the proof.
\end{proof}

\section{Conclusion}
In this note, we have obtained a wide range of upper bounds for the
moduli of the zeros of a complex polynomial. The bounds are summarized in
Theorem \ref{mainresult}, which completes some other known
results on the location of zeros of polynomials.

Finally, we point out that in order to compute the bounds it may be
preferable to use Theorem \ref{mainresult2} instead of Theorem
\ref{mainresult}, because it clarifies and simplifies the auxiliary
equation to be solved (compare expressions (\ref{Ql}) and
(\ref{Pl})). For example, when $\ell=2$, it takes the simple form
$x^{2}-(|a_{1}|+1)x-(A_{2}-|a_{1}|)=0$, and it gives rise to the
bound by Joyal et al. \cite{JLR1967}. For $\ell=3$ and $\ell=4$, we
obtain
\begin{equation*}
x^{3}-(|a_{1}|+1)x^{2}-(|a_{2}|-|a_{1}|)x-(A_{3}-|a_{2}|)=0
\end{equation*}
with $A_{3}=\max_{j\geq 3}|a_{j}|$, and
\begin{equation*}
x^{4}-(|a_{1}|+1)x^{3}-(|a_{2}|-|a_{1}|)x^{2}-(|a_{3}|-|a_{2}|)x-(A_{4}-|a_{3}|)=0
\end{equation*}
with $A_{4}=\max_{j\geq 4}|a_{j}|$. These equations can be
explicitly solved as cubic and quartic equations respectively.

\end{document}